\documentclass[reqno, a4paper]{amsart}
\usepackage{fixltx2e}
\usepackage{bookmark}
\usepackage{etoolbox}
\usepackage[usenames, dvipsnames]{color}

\definecolor{l-gray}{gray}{0.5}
\usepackage[english]{babel}
\usepackage{enumitem}

\usepackage[utf8]{inputenc}  
\usepackage[T1]{fontenc}
\usepackage{pifont}
\usepackage{gensymb}

\usepackage{rotating}

\usepackage{amsmath}
\usepackage{amssymb}
\usepackage{amsthm, remreset}
\usepackage{amscd}
\usepackage{centernot}
\usepackage{mathdots}
\usepackage[all,color,cmtip]{xy}

\makeatletter
\@namedef{subjclassname@2020}{%
  \textup{2020} Mathematics Subject Classification}
\makeatother

\newdir{ >}{{}*!/-7pt/\dir{>}}
\newdir{->-}{\dir{-}*\dir{>}*\dir{-}}
\newdir{  |}{{}*!/-20pt/\dir{|}}
\newdir{ |}{{}*!/-7pt/\dir{|}}
\newdir{|- }{!<2.5pt,0pt>\dir{|}*\dir{-}*{}}
\newdir{> }{!/7pt/\dir{>}*{}}
\newdir{>}{{}*:(1,-.2)@^{>}*:(1,+.2)@_{>}}
\newdir{<}{{}*:(1,+.2)@^{<}*:(1,-.2)@_{<}}

\newdir{>>}{{}*!/3.5pt/:(1,-.2)@^{>}*!/3.5pt/:(1,+.2)@_{>}*!/7pt/:(1,-.2)@^{>}*!/7pt/:(1,+.2)@_{>}}
\newdir{ >>}{{}*!/8pt/@{|}*!/3.5pt/:(1,-.2)@^{>}*!/3.5pt/:(1,+.2)@_{>}}
\newdir{ |>}{{}*!/-3.5pt/@{|}*!/-8pt/:(1,-.2)@^{>}*!/-8pt/:(1,+.2)@_{>}}
\newdir{ >}{{}*!/-8pt/@{>}}

\def\pullback{
 \ar@{-}[]+R+<3pt,-1pt>;[]+RD+<3pt,-3pt>%
 \ar@{-}[]+D+<1pt,-3pt>;[]+RD+<3pt,-3pt>}
 
 \def\lowpullback{
 \ar@{-}[]+R+<6pt,-6pt>;[]+D+<15pt,-16pt>;[]+RD+<6pt,-6pt>%
 \ar@{-}[]+D+<4pt,-10pt>;[]+RD+<6pt,-16pt>}

\def\pullbackdots{%
 \ar@{.}[]+R+<6pt,-1pt>;[]+RD+<6pt,-6pt>%
 \ar@{.}[]+D+<1pt,-6pt>;[]+RD+<6pt,-6pt>}
 
\def\pushout{%
 \ar@{-}[]+L+<-5pt,1pt>;[]+LU+<-5pt,5pt>%
 \ar@{-}[]+U+<-1pt,5pt>;[]+LU+<-5pt,5pt>}
 
\def\uppushout{%
 \ar@[Gray]@{-}[]+L+<-5pt,-1pt>;[]+LD+<-5pt,-5pt>%
 \ar@[Gray]@{-}[]+D+<-1pt,-5pt>;[]+LD+<-5pt,-5pt>}
 
 \def\lowuppushout{%
 \ar@{-}[]+L+<-10pt,-2pt>;[]+LD+<-10pt,-10pt>%
 \ar@{-}[]+D+<-2pt,-10pt>;[]+LD+<-10pt,-10pt>}

\def\splitpullback{%
 \ar@{-}[]+R+<6pt,-.51ex>;[]+RD+<6pt,-6pt>%
 \ar@{-}[]+D+<.51ex,-6pt>;[]+RD+<6pt,-6pt>}

\def\skewpullback{%
 \ar@{-}[]+R+<6pt,-1pt>;[]+RD+<-7pt,-8pt>%
 \ar@{-}[]+D+<-12pt,-8pt>;[]+RD+<-7pt,-8pt>}

\def\dottedpullback{%
 \ar@{.}[]+R+<6pt,-1pt>;[]+RD+<6pt,-6pt>%
 \ar@{.}[]+D+<1pt,-6pt>;[]+RD+<6pt,-6pt>}

\usepackage{geometry}
\let\origappendix\appendix 
\renewcommand\appendix{\clearpage\pagenumbering{roman}\origappendix}

\usepackage{graphicx}
\usepackage{epstopdf}
\newtheorem{theorem}[subsection]{Theorem}
	
\newtheorem{observation}[subsection]{Observation}	
	
\newtheorem{proposition}[subsection]{Proposition}	
	
\newtheorem{lemma}[subsection]{Lemma}
\newtheorem{definition}[subsection]{Definition}

\newtheorem{remark}[subsection]{Remark}

\newtheorem{algorithm}[subsection]{Algorithm}

\newcommand{\N}{\mathbb{N}}

\newcommand{\SET}{\mathsf{Set}}

\newcommand{\GRP}{\mathsf{Grp}}

\newcommand*{\defeq}{\mathrel{\vcenter{\baselineskip0.5ex \lineskiplimit0pt
                     \hbox{\scriptsize.}\hbox{\scriptsize.}}}%
                     =}

\DeclareMathOperator{\Fg}{F_g}

\DeclareMathOperator{\id}{id}

\DeclareMathOperator{\Eq}{Eq}

\DeclareMathOperator{\Pth}{Pth}

\DeclareMathOperator{\Ker}{Ker}

\usepackage{comment}
\includecomment{comment}
\definecolor{britishracinggreen}{rgb}{0.0, 0.26, 0.15}
\definecolor{darkpink}{rgb}{0.91, 0.33, 0.5}

\newcommand{\K}[1]{{\rm K}_{#1}}

\newcommand{\Test}[1]{{\rm TestPair(}#1{\rm )}}

\begin{document}

\title[The intersection of the kernels of two morphisms between free groups]{Rewriting the elements in the intersection of the kernels of two morphisms between free groups}
\author{Fran\c{c}ois Renaud}

\email[Fran\c{c}ois Renaud]{francois.renaud@uclouvain.be}

\address{Institut de Recherche en Math\'ematique et Physique, Universit\'e catholique de Louvain, che\-min du cyclotron~2 bte~L7.01.02, B--1348 Louvain-la-Neuve, Belgium}

\thanks{The author is a Ph.D.\ student funded by \textit{Formation à la Recherche dans l'Industrie et dans l'Agriculture} (\textit{FRIA}) as part of \textit{Fonds de la Recherche Scientifique - FNRS}}

\begin{abstract}
Let $\Fg$ be the free group functor, left adjoint to the forgetful functor between the category of groups $\GRP$ and the category of sets $\SET$. Let $f\colon A \to B$ and $h\colon A \to C$ be two functions in $\SET$ and let $\Ker(\Fg(f))$ and $\Ker(\Fg(h))$ be the kernels of the induced morphisms between free groups. Provided that the kernel pairs $\Eq(f)$ and $\Eq(h)$ of $f$ and $h$ permute (such as it is the case when the pushout of $f$ and $h$ is a double extension in $\SET$), this short article describes a method to rewrite a general element in the intersection $\Ker(\Fg(f)) \cap \Ker(\Fg(g))$ as a product of generators in $A$ which is $\langle f,h \rangle$-symmetric in the sense of the higher covering theory of racks and quandles.
\end{abstract}
\maketitle

\section{Motivations and context of study}\label{SectionMotivationsAndContext}
Our interest in the ``shape'' of a general element in the intersection of the kernels of two morphisms between free groups arises from research about the concept of a double covering of racks and quandles \cite{Ren2020,Ren2021}. Understanding this research, its context and its motivations is not strictly necessary for understanding this paper's result. However, the reader might appreciate some details about the context in which our result arises.

\emph{Quandles} \cite{Joy1979}, and the more general \emph{racks} (J.C.~Conway and G.C.~Wraith in 1959, see \cite{FenRou1992}), are algebraic structures that capture a certain notion of symmetry, which is used for instance in knot theory \cite{Joy1982,Bri1988,FenRou1992,Dri1992,FeRoSa1995,CaKaSa2001,CJKLS2003,Kau2012,Eis2014}, geometry \cite{Loos1969,Ber2008,Hel2012}, and outside of mathematics in computer science and physics (see for instance \cite{FenRou1992} for references). One of the original motivations for the study of such structures comes from the interest in the algebraic structure obtained from a group when the operation of multiplication is ``forgotten'', and the only operation that remains is \emph{group conjugation}. The links between the theory of racks and quandles on the one hand, and the theory of groups on the other, are tight. In particular, the development of a higher covering theory of racks and quandles started in \cite{Ren2020,Ren2021} is to be viewed in parallel with corresponding breakthroughs in group theory. Based on the development and applications of the categorical Galois theory of group, \cite{Jan1990,Jan1991,JanComo1991,JK1994}, the classical and higher \emph{Hopf formulae} \cite{EilMac1943,BroEll1988} were rediscovered (see \cite{Jan2008}) and generalized as in \cite{EvGrVd2008}. As it is explained in the introduction of \cite{Ren2020} and in the conclusions of \cite{Ren2021}, drawing the links between the homology of racks and quandles and the higher covering theory of these, via the study of generalized Hopf formulae \cite{EvGrVd2008,Jan2008} is one of the possible outcomes of such developments. Besides the links with group theory, categorical Galois theory also explains the parallels between the covering theory of racks and quandles and the covering theory of topological spaces. By extension, the development of a homotopy theory in racks and quandles is an interesting potential application.

The first step which we took towards these outcomes, was the definition and study of \emph{double coverings} of racks and quandles \cite{Ren2020}. These are special commutative squares of surjective morphisms of racks (or quandles) which we characterized (roughly speaking) by the ``trivial action'' of certain ``related'' \emph{double extensions} of groups.

More precisely, given a commutative square of surjective group homomorphisms
\begin{equation}\label{DiagramReflectionSquare} \vcenter{\xymatrix @R=0.8pt @ C=11pt{
G_1 \ar[rrr]^{f_1} \ar[rd] |-{p}  \ar[dddd]_{f_G} & & &  H_1 \ar[dddd]|{f_H} \\
 & P  \ar[rru]|-{\pi_2} \ar[lddd]|-{\pi_1}  \\
\\
\\
G_0  \ar[rrr]_{f_0} &  & & H_0
}}
\end{equation}
where $P$ is the subgroup of the product $G_0 \times H_1$ given by the pairs $(g,h)$ such that $f_0(g) = f_H(h)$ (i.e.~$P$ is the \emph{pullback} \cite{McLane1997} of $f_0$ and $f_H$), then the square \eqref{DiagramReflectionSquare} is said to be a \emph{double extension of groups} if the induced map $p\colon G_1 \to P \colon g \mapsto (f_G(g),f_1(g))$ is a surjection -- here $\pi_1$ and $\pi_2$ are the usual product projections. Define double extensions in $\SET$ similarly by replacing all surjective group homomorphisms in the above by surjective functions between sets -- double extensions of racks and quandles are defined in the same way. This concept of double extension is used (for instance in the covering theories of interest --see also \cite{Ever2010}) as a two-dimensional equivalent of a surjection (or more precisely \emph{regular epimorphism} \cite{McLane1997}).

Then, given a double extension of racks (or quandles), say $\alpha$, we may study its image in $\GRP$ under a certain \emph{functor}, i.e.~we look at an ``induced'' double extension $\Pth(\alpha)$ of the form \eqref{DiagramReflectionSquare}, where each group $G_1$, $H_1$, $G_0$ and $H_0$ is generated (but not necessarily freely generated) by the elements in the underlying sets of $\alpha$. We then identify ``special products of such generators'' in $G_1$, which are called \emph{$\langle f_1, f_G \rangle$-symmetric} (see Definition \ref{DefinitionK2}), and which are required to ``act trivially on $\alpha$'' for $\alpha$ to be a double covering. The set of such \emph{$\langle f_1, f_G \rangle$-symmetric} products of generators in $G_1$ form a normal subgroup of $G_1$. We are interested in characterizations of this subgroup, for instance in terms of the intersection of the kernels of $f_1$ and $f_G$.

Now we observe that $\Pth(\alpha)$ can be presented as a canonical quotient ($\Fg( \alpha) \twoheadrightarrow \Pth(\alpha)$) of the double extension $\Fg( \alpha)$ of the form \eqref{DiagramReflectionSquare}, obtained by taking \emph{free} groups on the underlying sets of $\alpha$, i.e.~such that $G_1$, $H_1$, $G_0$ and $H_0$ are free groups (on the underlying sets of $\alpha$). The result of this paper mainly consists in establishing that, in such a double extension of free groups $\Fg(\alpha)$, the elements in $\Ker(p)$, the kernel of $p$, which coincides with the intersection $\Ker(f_1) \cap \Ker(f_G)$ of the kernels of $f_1$ and $f_G$, are exactly those elements which are $\langle f_1, f_G \rangle$-symmetric in $G_1$. In order to extend this result to the double extension $\Pth(\alpha)$, it remains to describe what is preserved by the canonical presentation $\Fg( \alpha) \twoheadrightarrow \Pth(\alpha)$ (which goes beyond the scope of this article).


\section{Results}
\subsection{Preliminaries}\label{SectionPreliminaries}
Recall that given two groups $G$ and $H$ (written multiplicatively), the kernel $\Ker(h)$ of a group homomorphism $h\colon{G \to H}$ is the set of elements $a\in G$ which are sent by $h$ to the neutral element $e$ in $H$. Then $\Ker(h)$ is always a normal subgroup of $G$, i.e.~it contains the neutral element $e$ in $G$, it is closed under products and inverses, and it is closed under conjugation by elements of $G$: for $a \in \Ker(h)$ and $x \in G$ we have $x a x^{-1} \in \Ker(h)$.

Given a set $A$, the free group on $A$ is characterised as the group $\Fg(A)$ whose elements are equivalence classes of words on the alphabet $A \cup A^{-1}$ -- i.e.~sequences of elements $a_1^{\delta_1} \cdots a_n^{\delta_n}$ of length $n$, where $n$ is a natural number, $a_i \in A$ and $\delta_i$ is an exponent $1$ or $-1$, for each index $i$ with $1\leq i \leq n$ (we equivalently write $a$ for $a^1$). Two such words are identified if one can be obtained from the other by removing or adding pairs of letters $a a^{-1}$ or $a^{-1}a$ anywhere in the words, see also Observation \ref{ObservationPairingElementsInKernelWords} below. The product of two equivalence classes of words is defined as the equivalence class of the juxtaposition of any two words representing the factors. The neutral element is given by the equivalence class of the empty word $\emptyset$. A function $f\colon{A \to B}$ between two sets induces a group homomorphism $\Fg(f)$ between the free groups $\Fg(A)$ and $\Fg(B)$ which sends an element of $\Fg(A)$, represented by the word $a_1^{\delta_1} \cdots a_n^{\delta_n}$ to the element of $\Fg(B)$, represented by the word $f(a_1)^{\delta_1} \cdots f(a_n)^{\delta_n}$.

The kernel pair of $f$ in $\SET$ is the equivalence relation $\Eq(f) \subseteq A \times A$ defined as the subset of those pairs $(a,b) \in A \times A$ such that $f(a)=f(b)$. Similarly, the kernel pair of the group homomorphism $f \colon G \to H$ in $\GRP$ is given by the kernel pair of the function $f$ in $\SET$, equipped with the usual group structure induced by the group structure in each component.

In order to understand the hypotheses of our result (Lemma \ref{LemmaKernelOfFgFFgH}) we further recall that given two functions $f\colon {A \to B}$ and $h \colon{A \to C}$, the composite $\Eq(f) \circ \Eq(h)$ is defined as the subset of $A \times A$ given by the pairs $(a,b)$ such that there exists $c \in A$ for which $(a,c) \in \Eq(f)$ and $(c,b) \in \Eq(h)$. In general, $\Eq(f) \circ \Eq(h) \neq \Eq(h) \circ \Eq(f)$. However if $f$ and $h$ are the \emph{initial maps} of a double extension, such as $f_1$ and $f_G$  in \eqref{DiagramReflectionSquare}, then we have $\Eq(f) \circ \Eq(h) = \Eq(h) \circ \Eq(f)$ (see \cite[Proposition 5.4]{CKP1993} and \cite[Lemma 1.2]{Bou2003}).

More explicitely, the requirement $\Eq(f) \circ \Eq(h) = \Eq(h) \circ \Eq(f)$ means that for each $a$, $b$ and $c$ as above, there exists $d \in A$ such that $(a,d) \in \Eq(h)$ and $(d,b) \in \Eq(f)$, and conversely given such $a$, $b$ and $d$, one must be able to find some $c$ satisfying the above (i.e.~$f(a) = f(c)$ and $h(b) = h(c)$). In other words, one must always be able to ``complete'' triples such as 
\begin{equation}\label{EquationTriples} \vcenter{\xymatrix @C= 15pt@R=10pt {
a \ar@{-}[r]|-{f}   & c  \ar@{..}[d]|-{h} \\
  & b
}} \qquad \text{and} \qquad \vcenter{\xymatrix @C= 15pt@R=10pt {
a \ar@{..}[d]|-{h}  & \\
d \ar@{-}[r]|-{f}  & b
}} \qquad \text{into a full square}\qquad \vcenter{\xymatrix @C= 15pt@R=10pt {
a \ar@{-}[r]|-{f}  \ar@{..}[d]|-{h}  & c  \ar@{..}[d]|-{h} \\
d \ar@{-}[r]|-{f}  & b
}}\end{equation} 
where vertical lines (labelled by $h$) link elements which are identified by $h$ and horizontal lines (labelled by $f$) link elements which are identified by $f$. The equivalence relations $\Eq(f)$ and $\Eq(h)$ are then said to \emph{commute}  \cite{MalSbo1954,CKP1993}. 

Note that in the category of groups, two congruences on a group always commute in this sense (see \emph{Mal'tsev categories} \cite{CKP1993}). As a consequence (\cite[Proposition 5.4]{CKP1993}), the double extensions of groups coincide with \emph{pushout} \cite{McLane1997} squares of surjective group homomorphisms. In the category of sets, however, equivalence relations do not commute in general. Still, the kernel pairs of initial maps in a double extension of sets do commute. Considering the context of study described in Section \ref{SectionMotivationsAndContext}, this is why the requirement on commuting kernel pairs arises naturally in the hypotheses of Lemma \ref{LemmaKernelOfFgFFgH}.

In the following two sections, we recall the concept of \emph{symmetric paths}, which are the ``products of generators with specific symmetric features'' of interest to us. We start with the one-dimensional version of these ideas and move to the two-dimensional context in Subsection \ref{SectionTwoDimensionalSymmetricPaths}.

\subsection{One dimensional symmetric paths}
In \cite{Ren2020}, we introduce the following concepts and results (see also \cite{BonSta2021}).

\begin{definition}\label{DefinitionK1}
Given a group homomorphism $f \colon{G \to H}$, and a chosen subset $A \subseteq G$, we define (implicitly with respect to $A$)
\begin{enumerate}[label=(\roman*)]
\item two elements $g_a$ and $g_b$ in $G$ are \emph{$f$-symmetric (to each other)} if there exists $n\in \N$ and a sequence of pairs $(a_1,b_1)$, $\ldots$, $(a_n,b_n)$ in $(A\times A)$, such that \[ f(a_i)=f(b_i), \quad g_a = a_1^{\delta_1} \cdots  a_n^{\delta_n}, \quad \text{ and } \quad g_b = b_1^{\delta_1} \cdots  b_n^{\delta_n},\]
for some $\delta_i \in \lbrace -1,\, 1\rbrace$, where $1\leq i \leq n$. Alternatively say that $g_a$ and $g_b$ are an \emph{$f$-symmetric pair}.
\item $\K{f}$ is the set of \emph{$f$-symmetric paths} defined as the elements $g \in G$ such that $g = g_a g_b^{-1}$ for some $g_a$ and $g_b \in G$ which are $f$-symmetric to each other.
\end{enumerate}
\end{definition}

Observe that the elements of $\K{f}$ are in the kernel of $f$. The idea is to understand when $\K{f}$ actually describes all the elements in the kernel of $f$. For instance if the chosen subset $A$ is the whole group, or if it contains $\Ker(f)$ (the kernel of $f$), then we easily derive that $\K{f} = \Ker(f)$. For a general $f\colon G \to H$ as in Definition \ref{DefinitionK1}, the condition $\K{f} = \Ker(f)$ expresses the fact that the kernel of $f$ in $\GRP$ is entirely described by the restriction $f \colon A \to f(A)$ of the underlying function $f$ in the category $\SET$. This is for instance the case when $f\colon G \to H = \Fg(h) \colon \Fg(A) \to \Fg(B)$ is the group homomorphism induced by a function $h \colon A \to B$ in $\SET$ as in Subsection \ref{SectionPreliminaries} (where the chosen subset of $G = \Fg(A)$ is $A$ -- see Proposition \ref{LemmaKernelOfFgF}). Even though in the examples of interest, the chosen subset $A$ is a generating set of $G$ (i.e.~such that the subgroup $\langle a \, \mid\, a \in A \rangle_G$ of $G$ generated by the elements of $A$ is equal to $G$), it is neither sufficient, nor necessary, for $A$ to be such a generating set of $G$ in general. For instance, consider the quotient map $f \colon \Fg(\lbrace a \rbrace) \to \lbrace e,\ a \rbrace$ where $f(a^n) = e$ if $n$ is even and $f(a^n) = a$ if $n$ is odd. The element $a^2$ is in $\Ker(f)$. However, if $\lbrace a \rbrace$ is our chosen set (of generators) of $\Fg(\lbrace a \rbrace)$, the element $a^2$ is not an $f$-symmetric path. Conversely, Definition \ref{DefinitionK1} and the condition $\K{f} = \Ker(f)$ still make sense when $A$ is merely a subset of $G$ which is not generating. For instance, consider the product $\Fg(h) \times \id_{G'} \colon \Fg(A) \times G' \to \Fg(B) \times G'$ of $\Fg(h) \colon \Fg(A) \to \Fg(B)$ with the identity function on some other group $G'$. Then $\Ker(\Fg(h) \times \id_{G'}) = \Ker(\Fg(h)) \times \lbrace e \rbrace = \K{\Fg(h) \times \id_{G'}}$ with chosen subset $A \times \lbrace e \rbrace$.

\begin{lemma}\label{LemmaGroupCharacterization1}
Given the hypotheses of Definition \ref{DefinitionK1}, the set of $f$-symmetric paths $\K{f} \subseteq G$ defines a normal subgroup in $G$. More precisely it is the normal subgroup generated by the elements of the form $ab^{-1}$ such that  $a$, $b \in A$, and $f(a)=f(b)$:
\[ \K{f} = G_f \defeq \langle\langle ab^{-1} \mid (a,b) \in A \times A,\ f(a)=f(b) \rangle \rangle_G. \]
\end{lemma}

\begin{observation}\label{ObservationPairingElementsInKernelWords}
Consider a function $f\colon {A \to B}$, and a word $\nu = a_1^{\delta_1} \cdots a_n^{\delta_n}$ with $a_i \in A$ and $\delta_i \in \lbrace -1,\, 1\rbrace$, for $1 \leq i \leq n$. This word represents an element $g$ in the free group $\Fg(A)$. As usual, a \emph{reduction} of $\nu$ consists in eliminating, in the word $\nu$, an adjacent pair $a_i^{\delta_i}a_{i+1}^{\delta_{i+1}}$ such that $\delta_{i} = -\delta_{i+1}$ and $a_{i} = a_{i+1}$. Every element $g \in \Fg(A)$ represented by a word $\nu$ admits a unique \emph{normal form} i.e.~a word $\nu'$ obtained from $\nu$ after a sequence of reductions, such that there is no reduction possible in $\nu'$, but $\nu'$ still represents the same element $g$ in $\Fg(A)$.

Suppose that $\nu$ represents an element $g$ which is in the kernel $\Ker(\Fg(f))$. The normal form of the word $f[\nu] \defeq f(a_1)^{\delta_1} \cdots f(a_n)^{\delta_n}$ (which represents $\Fg(f)(g) = e \in \Fg(B)$) is the empty word $\emptyset$, and thus there is a sequence of reductions of $f[\nu]$ such that the end result is $\emptyset$. From this sequence of reductions, we may deduce that $n=2m$ for some $m \in \N$ and the letters in the word (or sequence) $\nu$ organize themselves in $m$ pairs $(a_i^{\delta_i},a_j^{\delta_j})$ (the pre-images of those pairs that are reduced at some point in the aforementioned sequence of reductions) such that $i < j$, $f(a_i)= f(a_j)$, $\delta_i = - \delta_j$, each letter of the word $\nu$ appears in only one such pair and finally given any two such pairs $(a_i^{\delta_i},a_j^{\delta_j})$ and $(a_l^{\delta_l},a_m^{\delta_m})$, then $l < i$ (respectively $l> i$) if and only $m> j$ (respectively $m< j$),  i.e.~drawing lines which link those letters of the word $\nu$ that are identified by the pairing, none of these lines can cross. 
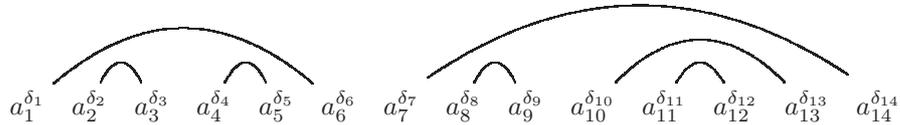
\[\xymatrix@R=20pt@C=4pt{ \\ 
a_1^{\delta_1}\ar@{{}{-}{}}@/^7ex/[rrrrr] & a_2^{\delta_2} \ar@{{}{-}{}}@/^4ex/[r] & a_3^{\delta_3} & a_4^{\delta_4} \ar@{{}{-}{}}@/^4ex/[r] & a_5^{\delta_5}  & a_6^{\delta_6} & a_7^{\delta_7} \ar@<1ex>@{{}{-}{}}@/^8ex/[rrrrrrr] & a_8^{\delta_8}\ar@{{}{-}{}}@/^4ex/[r]  & a_9^{\delta_9} & a_{10}^{\delta_{10}} \ar@{{}{-}{}}@/^6ex/[rrr] & a_{11}^{\delta_{11}} \ar@{{}{-}{}}@/^4ex/[r]  & a_{12}^{\delta_{12}} & a_{13}^{\delta_{13}} & a_{14}^{\delta_{14}}}\]
Given such a pairing of the letters of $\nu$, for each $k \in \lbrace 1,\ \ldots,\ n \rbrace$ we write $(a_{i_k}^{\delta_{i_k}},a_{j_k}^{\delta_{j_k}})$ for the unique pair such that either $i_k = k$ or $j_k=k$. Note that, conversely, any element $g$ in $\Fg(A)$ which is represented by a word $\nu$ which admits such a pairing of its letters, is necessarily in $\Ker(\Fg(f))$.
\end{observation}

Using this observation, we characterize the kernels of maps between free groups.

\begin{proposition}\label{LemmaKernelOfFgF}
Given a function $f\colon {A \to B}$, the kernel $\Ker(\Fg(f))$ of the induced group homomorphism $\Fg(f) \colon {\Fg(A) \to \Fg(B)}$ is given by the normal subgroup $\K{\Fg(f)}$ of $\Fg(f)$-symmetric paths (as in Definition \ref{DefinitionK1}): $\Ker(\Fg(f)) = \K{\Fg(f)}$.
\end{proposition}
\begin{proof}
The inclusion $\Ker(\Fg(f)) \supseteq \K{\Fg(f)}$ is obvious. Consider a reduced word $\nu = a_1^{\delta_1} \cdots a_n^{\delta_n} $ of length $n\in \N$ which represents an element $g$ in $\Fg(A)$ with $\delta_i \in \lbrace -1,\, 1\rbrace$, for $1 \leq i \leq n$ and suppose that $g \in \Ker(\Fg(f))$. Then the letters $a_k^{\delta_k}$ of the sequence (or word) $\nu \defeq (a_k^{\delta_k})_{1\leq k \leq n}$ organize themselves in pairs $(a_{i_k}^{\delta_{i_k}},a_{j_k}^{\delta_{j_k}})$ as in Observation \ref{ObservationPairingElementsInKernelWords}. Define the word $\nu' = b_1^{\delta_1} \cdots b_n^{\delta_n}$ such that for each $1 \leq k \leq n$, $b_k \defeq a_{i_k}$. Then by construction $\nu'$ represents an element $h$ which reduces to the empty word in $\Fg(A)$, so that $g = gh^{-1}$. Moreover, $g$ and $h$ form an $f$-symmetric pair, which shows that $g \in \K{\Fg(f)}$. \qedhere
\end{proof}

Note that in \cite{Ren2020} we provide another family of examples of group homomorphisms $f\colon G \to H$ such that $\K{f} = \Ker(f)$ where $f$ is the group homomorphism induced by a morphism of racks between the corresponding \emph{groups of paths}.

\subsection{Two-dimensional symmetric paths}\label{SectionTwoDimensionalSymmetricPaths}
In \cite{Ren2021}, we extend the concept of symmetric path as follows.

\begin{definition}\label{DefinitionK2}
Given a pair of morphisms $f \colon{G \to H}$, $h \colon{G \to K}$ in $\GRP$, and a generating set $A \subseteq G$ (i.e.~such that $G = \langle a \, \mid\, a \in A \rangle_G$), we define (implicitly \emph{with respect to $A$}):
\begin{enumerate}[label=(\roman*)]
\item four elements $g_a$, $g_b$, $g_c$ and $g_d$ in $G$ are \emph{$\langle f,h\rangle$-symmetric (to each other)} if there exists $n\in \N$ and a sequence of quadruples $(a_1,b_1,c_1,d_1)$, $\ldots$, $(a_n,b_n,c_n,d_n)$ in $A^4$ such that for each index $i$, 
\[ \vcenter{\xymatrix @C= 15pt@R=10pt {
a_i \ar@{-}[r]|-{f}  \ar@{..}[d]|-{h}  & b_i  \ar@{..}[d]|-{h} \\
d_i \ar@{-}[r]|-{f}  & c_i
}} \quad \text{i.e.}\ f(a_i) = f(b_i),\ f(d_i) = f(c_i),\ h(a_i) = h(d_i),\ h(b_i) = h(c_i);\]
and finally, for each $1\leq i \leq n$, there is $\delta_i \in \lbrace -1,\, 1\rbrace$ such that: 
\begin{equation}\label{EquationFHSymmetricQuadruple}
g_a = a_1^{\delta_1} \cdots  a_n^{\delta_n}, \quad g_b = b_1^{\delta_1} \cdots  b_n^{\delta_n}, \quad g_c = c_1^{\delta_1} \cdots  c_n^{\delta_n}, \quad g_d = d_1^{\delta_1} \cdots  d_n^{\delta_n}.
\end{equation}
We often call such $g_a$, $g_b$, $g_c$ and $g_d$ an \emph{$\langle f,h\rangle$-symmetric quadruple}.
\item $\K{\langle f,h\rangle}$ is the set of \emph{$\langle f,h\rangle$-symmetric paths}, i.e.~the elements $g \in G$ such that $g = g_a g_b^{-1} g_c g_d^{-1}$ for some $\langle f,h\rangle$-symmetric quadruple $g_a$, $g_b$, $g_c$ and $g_d \in G$.
\end{enumerate}
\end{definition}

\begin{lemma}\label{LemmaGroupCharacterization2}
Given the hypotheses of Definition \ref{DefinitionK2}, the set of $\langle f,h\rangle$-symmetric paths $\K{\langle f,h\rangle}$ defines a normal subgroup of $G$. \qed
\end{lemma}

Finally we state and prove the result we are interested in.

\begin{theorem}\label{LemmaKernelOfFgFFgH}
Given two surjective functions $f\colon {A \to B}$ and $h \colon{A \to C}$ such that $\Eq(f) \circ \Eq(h) = \Eq(h) \circ \Eq(f)$, the intersection $\Ker(\Fg(f)) \cap \Ker(\Fg(h))$ of the kernels of the induced group homomorphisms $\Fg(f) \colon {\Fg(A) \to \Fg(B)}$ and $\Fg(h) \colon {\Fg(A) \to \Fg(C)}$ is given by $\K{\langle \Fg(f),\Fg(h)\rangle}$ (with respect to $A$) as in Definition \ref{DefinitionK2}.
\end{theorem}
\begin{proof}
Given an element $g \in \Ker(\Fg(f)) \cap \Ker(\Fg(h)) \subseteq \Fg(A)$, and following the proof of Lemma \ref{LemmaKernelOfFgF} based on Observation \ref{ObservationPairingElementsInKernelWords}, we may identify an $\Fg(f)$-symmetric pair $g_a = a_1^{\delta_1} \cdots  a_n^{\delta_n}$, $g_b = b_1^{\delta_1} \cdots  b_n^{\delta_n}$ in $\Fg(A)$, such that $g = g_ag_b^{-1}$. Moreover, since $g \in \Ker(\Fg(h))$, by Observation \ref{ObservationPairingElementsInKernelWords}, the elements in the sequence (or word)
\[ (x_i^{\gamma_i})_{1\leq i \leq 2n} \defeq a_1^{\delta_1},\ \ldots,\  a_n^{\delta_n},\ b_n^{-\delta_n},\ \ldots,\ b_1^{-\delta_1}\]
organize themselves in $n$ pairs $(x_i^{\gamma_i},x_j^{\gamma_j})$ such that $i<j$, $(x_i, x_j) \in \Eq(h)$, $\gamma_i = - \gamma_j$, each element of the sequence $(x_i^{\gamma_i})_{1\leq i \leq 2n}$ appears in only one such pair, and given any two such pairs $(x_i^{\gamma_i},x_j^{\gamma_j})$ and $(x_l^{\gamma_l},x_m^{\gamma_l})$, $l< i$ (respectively $l > i$) if and only if $m > j$ (respectively $m<j$). Let us fix such a choice of pairs. For each $k \in \lbrace 1,\ \ldots,\ 2n \rbrace$, we write $(x_{i_k}^{\gamma_{i_k}},x_{j_k}^{\gamma_{j_k}})$ for the unique pair such that either $i_k = k$ or $j_k=k$. 

For what follows, consider the ``paired index'' operation $p$ defined by $p(i_k) \defeq j_k$ and $p(j_k) \defeq i_k$ for any $k\in \lbrace 1,\ \ldots,\ 2n \rbrace$ -- of course $p(p(k))=k$ for all such $k$. Define the operation ``opposite index'' $o(i) = 2n+1-i$ for $1 \leq i \leq 2n$.

We give an example of such a word representing $g \in \Fg(A)$ below, where $n = 7$, the lower lines link elements which are sent to opposites by $f$, and the upper lines link the elements which are paired -- and thus sent to opposites by $h$. We have for instance $(x_{i_{6}}^{\gamma_{i_6}},x_{j_6}^{\gamma_{j_6}}) = (x_{1}^{\gamma_{1}},x_{6}^{\gamma_{6}}) = (x_{i_{1}}^{\gamma_{i_1}},x_{j_1}^{\gamma_{j_1}})$ and similarly $(x_{i_{7}}^{\gamma_{i_7}},x_{j_7}^{\gamma_{j_7}}) = (x_{7}^{\gamma_{7}},x_{14}^{\gamma_{14}}) = (x_{i_{14}}^{\gamma_{i_{14}}},x_{j_{14}}^{\gamma_{j_{14}}}) $, moreover $o(n) = 8$, $p(n) = 14$ and $p(o(n)) = 9$, $o(p(n))=1$:
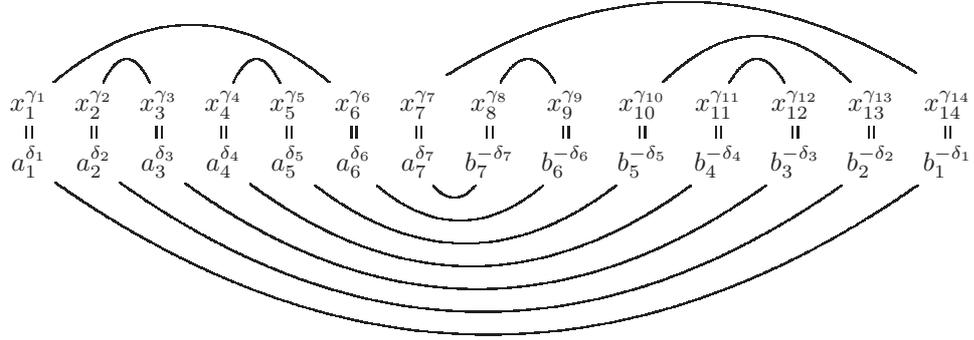
\[\xymatrix@R=5pt@C=4pt{ \\ 
\\
x_1^{\gamma_1}\ar@{{}{-}{}}@/^7ex/[rrrrr] \ar@{{}{=}{}}[d] & x_2^{\gamma_2} \ar@{{}{=}{}}[d] \ar@{{}{-}{}}@/^4ex/[r] & x_3^{\gamma_3} \ar@{{}{=}{}}[d] & x_4^{\gamma_4} \ar@{{}{=}{}}[d] \ar@{{}{-}{}}@/^4ex/[r] & x_5^{\gamma_5} \ar@{{}{=}{}}[d]  & x_6^{\gamma_6} \ar@{{}{=}{}}[d] & x_7^{\gamma_7} \ar@{{}{=}{}}[d] \ar@<1ex>@{{}{-}{}}@/^8ex/[rrrrrrr] & x_8^{\gamma_8} \ar@{{}{=}{}}[d] \ar@{{}{-}{}}@/^4ex/[r]  & x_9^{\gamma_9} \ar@{{}{=}{}}[d] & x_{10}^{\gamma_{10}} \ar@{{}{=}{}}[d] \ar@{{}{-}{}}@/^6ex/[rrr] & x_{11}^{\gamma_{11}} \ar@{{}{=}{}}[d] \ar@{{}{-}{}}@/^4ex/[r]  & x_{12}^{\gamma_{12}} \ar@{{}{=}{}}[d] & x_{13}^{\gamma_{13}} \ar@{{}{=}{}}[d] & x_{14}^{\gamma_{14}} \ar@{{}{=}{}}[d] \\
a_1^{\delta_1}\ar@{{}{-}{}}@/_15ex/[rrrrrrrrrrrrr] & a_2^{\delta_2} \ar@{{}{-}{}}@/_13ex/[rrrrrrrrrrr] & a_3^{\delta_3} \ar@{{}{-}{}}@/_11ex/[rrrrrrrrr]  & a_4^{\delta_4} \ar@{{}{-}{}}@/_9ex/[rrrrrrr]  & a_5^{\delta_5} \ar@{{}{-}{}}@/_7ex/[rrrrr]   & a_6^{\delta_6} \ar@{{}{-}{}}@/_5ex/[rrr]  & a_7^{\delta_7} \ar@{{}{-}{}}@/_3ex/[r] & b_7^{-\delta_7} & b_6^{-\delta_6} & b_{5}^{-\delta_{5}} & b_{4}^{-\delta_{4}}  & b_{3}^{-\delta_{3}} & b_{2}^{-\delta_{2}} & b_{1}^{-\delta_{1}} \\
\\
\\
\\
\\
\\
}\]

We may rewrite the word representing $g$ as the word 
\[a_1^{\delta_1} \cdots  a_n^{\delta_n} b_n^{-\delta_n} \cdots  b_1^{-\delta_1} y_{2n}^{\delta_{1}} \cdots  y_{n+1}^{\delta_n} y_n^{-\delta_n} \cdots  y_1^{-\delta_1},\]
where the $y_k$ are systematically chosen to be $x_{i_{k}}$ as in the pair $(x_{i_k}^{\gamma_{i_k}},x_{j_k}^{\gamma_{j_k}})$. If we define $\sigma_k \defeq -\delta_k$ for $1\leq k \leq n$ and $\sigma_k \defeq \delta_{2n +1 - k}$ for $n+1 \leq k \leq 2n$, then $y_{p(k)} = y_k$ and  $\sigma_{p(k)} = - \sigma_k$ for all $k\in \lbrace 1,\ \ldots,\ 2n \rbrace$, and the sequence (or word) $(y_k^{-\sigma_k})_{1 \leq k \leq 2n}$ reduces to the empty word, i.e.~it represents the neutral element $e$ in $\Fg(A)$. In our example

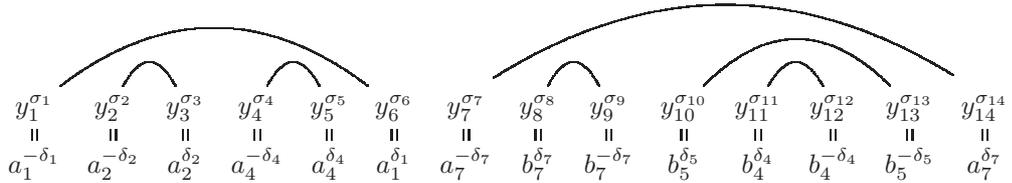
\[\xymatrix@R=5pt@C=4pt{ \\
\\
y_1^{\sigma_1}\ar@{{}{-}{}}@/^7ex/[rrrrr] \ar@{{}{=}{}}[d] & y_2^{\sigma_2} \ar@{{}{=}{}}[d] \ar@{{}{-}{}}@/^4ex/[r] & y_3^{\sigma_3} \ar@{{}{=}{}}[d] & y_4^{\sigma_4} \ar@{{}{=}{}}[d] \ar@{{}{-}{}}@/^4ex/[r] & y_5^{\sigma_5} \ar@{{}{=}{}}[d]  & y_6^{\sigma_6} \ar@{{}{=}{}}[d] & y_7^{\sigma_7} \ar@{{}{=}{}}[d] \ar@<1ex>@{{}{-}{}}@/^8ex/[rrrrrrr] & y_8^{\sigma_8} \ar@{{}{=}{}}[d] \ar@{{}{-}{}}@/^4ex/[r]  & y_9^{\sigma_9} \ar@{{}{=}{}}[d] & y_{10}^{\sigma_{10}} \ar@{{}{=}{}}[d] \ar@{{}{-}{}}@/^6ex/[rrr] & y_{11}^{\sigma_{11}} \ar@{{}{=}{}}[d] \ar@{{}{-}{}}@/^4ex/[r]  & y_{12}^{\sigma_{12}} \ar@{{}{=}{}}[d] & y_{13}^{\sigma_{13}} \ar@{{}{=}{}}[d] & y_{14}^{\sigma_{14}} \ar@{{}{=}{}}[d] \\
a_1^{-\delta_1}  & a_2^{-\delta_2} & a_2^{\delta_2}  & a_4^{-\delta_4}  & a_4^{\delta_4}   & a_1^{\delta_1}  & a_7^{-\delta_7}  & b_7^{\delta_7} & b_7^{-\delta_7} & b_{5}^{\delta_{5}}  & b_{4}^{\delta_{4}} & b_{4}^{-\delta_{4}}  & b_{5}^{-\delta_{5}}  & a_{7}^{\delta_{7}} 
}\]
Observe moreover that for all $i \in \lbrace 1,\ \ldots,\ 2n \rbrace$, $(x_i,y_i) \in \Eq(h)$, $\gamma_i = - \sigma_i$ and $\sigma_i = -\sigma_{2n+1-i}$. In order for $g$ to be in $\K{2}(\Fg(f),\Fg(h))$, it then suffices to check that $(y_i, y_{2n+1-i}) \in \Eq(f)$ for each $i \in \lbrace 1,\ \ldots,\ n\rbrace$. Since this is not the case in general, we describe how to ``algorithmically'' replace the value of each $y_i$ ($y_i \mapsto d_i$ for $1 \leq i \leq n$ and $y_i \mapsto c_{2n+1-i}$ for $n+1 \leq i \leq 2n$) in such a way that for each index $1\leq i \leq 2n$, it is still the case that $y_{p(i)} = y_i$, moreover, the previous and new value of $y_i$ are identified by $h$ and finally $(c_k,d_k) \in \Eq(f)$ for each $k \in \lbrace 1,\ \ldots,\ n\rbrace$. The resulting word below still represents $g$ which thus satisfies all the conditions for being an element of $\K{2}(\Fg(f),\Fg(h))$.
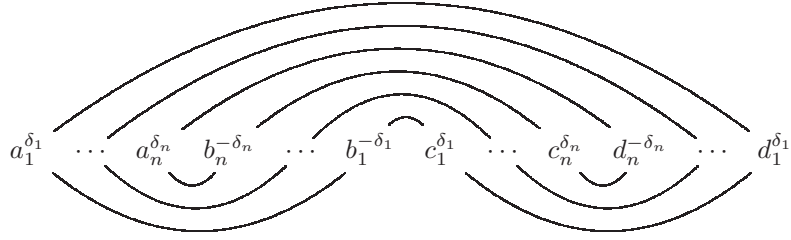
\[ \xymatrix@R=20pt @C=5pt{\\
\\
a_1^{\delta_1} \ar@{{}{-}{}}@/_7ex/[rrrrr] \ar@{{}{-}{}}@/^13ex/[rrrrrrrrrrr] & \cdots \ar@{{}{-}{}}@/_5ex/[rrr] \ar@{{}{-}{}}@/^11ex/[rrrrrrrrr]  & a_n^{\delta_n} \ar@{{}{-}{}}@/_3ex/[r] \ar@{{}{-}{}}@/^9ex/[rrrrrrr] & b_n^{-\delta_n} \ar@{{}{-}{}}@/^7ex/[rrrrr] & \cdots \ar@{{}{-}{}}@/^5ex/[rrr] & b_1^{-\delta_1} \ar@{{}{-}{}}@/^3ex/[r] & c_1^{\delta_1} \ar@{{}{-}{}}@/_7ex/[rrrrr] & \cdots \ar@{{}{-}{}}@/_5ex/[rrr] & c_n^{\delta_n} \ar@{{}{-}{}}@/_3ex/[r] & d_n^{-\delta_n} & \cdots &  d_1^{\delta_1}\\
\\
} \] 

Let us illustrate this rewriting method on our example. The general method is then described below in Algorithm \ref{Algorithm}. We start by looking at the pair $(y_7,y_8) = (a_7,b_7)$ and we observe that it is in $\Eq(f)$. Hence we do not change the values of $y_7$ and $y_8$, i.e.~define $d_7 \defeq a_7$ and $c_7 \defeq b_7$ respectively. Now once we have set the values of $d_7$ and $c_7$, we do not want to modify these anymore. But remember that $y_{p(7)} = y_{14}$ has to be equal to $y_{7} = d_7$ and similarly $y_{p(o(7))} = y_{p(8)} = y_9$ has to be equal to $y_8 = c_7$. We thus fix $y_9 = c_6 \defeq  c_7 = b_7$ and $y_{14} = c_1 \defeq d_7 = a_7$. We then proceed by looking at the pair $(y_{p(o(7))},y_{o(p(o(7)))}) = (y_9,y_6) = (b_7,a_1)$ which is not known to be in $\Eq(f)$. We are then going to modify the value of $y_{6}$ accordingly, since the value of $y_9$ is set. Observe that since
\begin{equation}\label{EquationZ}
\vcenter{\xymatrix@R=10pt@C=4pt{
y_9 = c_6 = b_7  \ar@{{}{-}{}}@/^4ex/[r]|-{h} & x_9 = b_6 \ar@{{}{-}{}}@/_4ex/[r]|-{f} & x_{6} = a_6 \ar@{{}{-}{}}@/^4ex/[r]|-{h} & y_{6} = a_1,
}} \quad \text{there is } z \in A\text{ such that} \quad
\vcenter{\xymatrix@R=10pt@C=4pt{
y_9 \ar@{{}{-}{}}@/_4ex/[r]|-{f} & z \ar@{{}{-}{}}@/^4ex/[r]|-{h} & x_{6} \ar@{{}{-}{}}@/^4ex/[r]|-{h} & y_{6},
}}
\end{equation} 
where we link two elements with a line labelled by $h$ if these are identified by $h$; similarly for $f$.

We then let $d_6 \defeq z$ be the new value of $y_{6}$. Then $y_{p(6)} = y_1$ has to be set to $d_1 \defeq d_6$. We should then look at the pair $(y_1,y_{14})$, but both values $d_1$ and $c_1$ have been defined already. This is not a problem since $(y_1,y_{14}) \in \Eq(f)$ by construction ($f(y_1) = f(d_6) = f(c_6) = f(c_7)= f(c_1) = f(y_{14})$).

We have just completed an \emph{Inner loop} in our method. We may then start over by choosing any of the remaining indices $i$, such that the value of $y_i$ has not been set yet. Take for instance $i=2$, we look at the pair $(y_2,y_{13}) = (a_2,b_5)$ which is not known to be in $\Eq(f)$. We thus modify the value of $y_{13}$ accordingly:
since
\[\vcenter{\xymatrix@R=10pt@C=4pt{
y_2 = a_2  \ar@{{}{-}{}}@/^4ex/[r]|-{h} & x_2 = a_2 \ar@{{}{-}{}}@/_4ex/[r]|-{f} & x_{13} = b_2 \ar@{{}{-}{}}@/^4ex/[r]|-{h} & y_{13} = b_5,
}} \quad \text{there is } z' \in A\text{ such that} \quad
\vcenter{\xymatrix@R=10pt@C=4pt{
y_2 \ar@{{}{-}{}}@/_4ex/[r]|-{f} & z' \ar@{{}{-}{}}@/^4ex/[r]|-{h} & x_{13} \ar@{{}{-}{}}@/^4ex/[r]|-{h} & y_{13},
}} \]
we choose $y_{13} = c_2 \defeq z'$ and keep $y_2 = d_2 \defeq a_2$. As a consequence we define $y_{p(2)} = y_3 = d_3 \defeq a_2$ and $y_{p(13)} = y_{10} = c_5 \defeq c_2$. We must then look at $(y_{10},y_5) = (c_2, a_4)$ which is not known to be in $\Eq(f)$. Then since 
\[\vcenter{\xymatrix@R=10pt@C=4pt{
y_{10} = c_2  \ar@{{}{-}{}}@/^4ex/[r]|-{h} & x_{10} = b_5 \ar@{{}{-}{}}@/_4ex/[r]|-{f} & x_{5} = a_5 \ar@{{}{-}{}}@/^4ex/[r]|-{h} & y_{5} = a_4,
}} \quad \text{there is } z'' \in A\text{ such that} \quad
\vcenter{\xymatrix@R=10pt@C=4pt{
y_{10} \ar@{{}{-}{}}@/_4ex/[r]|-{f} & z'' \ar@{{}{-}{}}@/^4ex/[r]|-{h} & x_{5} \ar@{{}{-}{}}@/^4ex/[r]|-{h} & y_{5},
}} \]
we define $y_{5} = d_5 \defeq z''$ and thus $y_{p(5)} = y_{4} = d_4 \defeq d_5$. Similarly, we define $y_{11} = c_4$ such that $(y_4, y_{11}) \in \Eq(f)$ and let $y_{12} = c_3 \defeq c_4$. Now observe that $y_3 = d_3 = a_2$ was defined in such a way that $(y_{12},y_3) \in \Eq(f)$. This ends a second \emph{Inner loop} in the method. The values of all the $y_i$ have been set in the appropriate way, which ends the \emph{Outer loop} of our method. In general this method can be implemented as in Algorithm \ref{Algorithm} below.

\begin{algorithm}\label{Algorithm}
Declare two variables $m$ and $l$, which range over the indices $\lbrace 1,\ \ldots,\ 2n\rbrace$, and which we use to run the two embedded loops of the algorithm. Define the variable $I$ which is the set of ``indices not yet visited''. We use the symbol $\defeq$ to change the value of these variables. 

Start by setting $m \defeq n$ and $I \defeq \lbrace 1,\ \ldots,\ 2n\rbrace$.

\subsubsection*{Outer loop}
Define $I \defeq I \setminus \lbrace n,\ o(m) \rbrace$ and start by running $\Test{m}$ defined below. Then:
\begin{enumerate}
\item If $p(m) \neq o(m)$, replace the values of $y_{p(m)}$ and $y_{p(o(m))}$ by those of $y_{m}$ and $y_{o(m)}$ respectively. Define $l \defeq p(o(m))$ and proceed to the \textit{Inner loop}.
\item Otherwise (if $p(m) = o(m)$), proceed to \textit{Switch}.
\end{enumerate}

\subsubsection*{Inner loop}
Define $I \defeq I \setminus \lbrace l,\ o(l)\rbrace$ and then:
\begin{enumerate}
\item If $o(l) \neq p(m)$, then the value of $y_{o(l)}$ has not been modified yet. Run $\Test{l}$, and replace the value of $y_{p(o(l))}$ by that of $y_{o(l)}$. Observe that $y_{p(m)} = y_{m}$, $y_{o(m)} = y_{p(o(m))}$, $y_l$, $y_{o(l)} = y_{p(o(l))}$ are all identified by $f$ by construction. Finally, redefine $l \defeq p(o(l))$ and proceed to the \textit{Inner loop}.
\item Otherwise ($o(l) = p(m)$), proceed to \textit{Switch}.
\end{enumerate}

\subsubsection*{Switch} 
\begin{enumerate} 
\item If $I$ is not empty, choose any $i \in I$ and define $m \defeq i$. Proceed to the \textit{Outer loop}.
\item Otherwise stop the algorithm.
\end{enumerate}

\subsubsection*{TestPair} 
Given an index $j \in \lbrace 1,\ \ldots,\ 2n\rbrace$, define the method $\Test{j}$ as follows.
\begin{enumerate}
\item If $(y_j,y_{o(j)}) \in \Eq(f)$, keep these as they are;
\item otherwise, since
\[\vcenter{\xymatrix@R=10pt@C=4pt{
y_j \ar@{{}{-}{}}@/^4ex/[r]|-{h} & x_j \ar@{{}{-}{}}@/_4ex/[r]|-{f} & x_{o(j)} \ar@{{}{-}{}}@/^4ex/[r]|-{h} & y_{o(j)},
}} \quad \text{there exists } z \in A \text{ such that} \quad
\vcenter{\xymatrix@R=10pt@C=4pt{
y_j \ar@{{}{-}{}}@/_4ex/[r]|-{f} & z \ar@{{}{-}{}}@/^4ex/[r]|-{h} & x_{o(j)} \ar@{{}{-}{}}@/^4ex/[r]|-{h} & y_{o(j)}.
}} \]
Then replace the value of $y_{o(j)}$ by $z$.
\end{enumerate}
\end{algorithm}

Note that $l$ and $m$ keep visiting new indices, which have not yet been visited by $l$, $o(l)$, $m$ or $o(m)$. Since there is a finite amount of such indices and $o(p(m))$ is one of those, the \emph{Inner loop} always reaches an end: $o(l) = p(m)$. When this happens, by construction $y_l$ and $y_{p(m)} = y_{o(l)}$ are in relation by $\Eq(f)$. We can thus proceed to the \emph{Outer loop} after redefining $m$ to be any of the remaining indices in $I$. No value of any $y_i$ is changed twice, but all values are visited once so that the resulting sequence is as required.
\end{proof}

\begin{remark} Note that the existence of a $z$ as in Equation \eqref{EquationZ} does not mean that there is a workable algorithm to find this $z$. However, assuming that such a procedure exists, one can implement the whole method on a computer.
\end{remark}

\bibliography{AlgorithmicAnecdotes.bib}
\bibliographystyle{amsplain-nodash}

\end{document}